\documentclass[reqno,11pt]{amsart}
\usepackage{amsmath,amsthm,amsfonts,amssymb,latexsym,epsfig,geometry,youngtab}
\usepackage{ytableau}
\usepackage{hyperref,times}
\newtheorem{theorem}{Theorem}[section]

\newtheorem{lemma}{Lemma}[section]
\newtheorem{corollary}{Corollary}[section]
\newtheorem{conj}{Conjecture}[section]
\theoremstyle{remark}
\theoremstyle{definition}

\numberwithin{equation}{section}

\makeatother
\makeatletter
\def\section{\@startsection{section}{1}
\z@{.7\linespacing\@plus\linespacing}{.5\linespacing}
{\normalfont\scshape}}
\makeatother

\begin{document}

\title[Partition-theoretic interpretation for certain truncated series]
{Partition-theoretic interpretation for certain truncated series}
\author{D. S. Gireesh$^1$ and B. Hemanthkumar$^1$}
\address{$^1$ Department of Mathematics, M. S. Ramaiah University of Applied Sciences, Peenya, Bengaluru-560 058, Karnataka, India}

\email{gireeshdap@gmail.com}\email{hemanthkumarb.30@gmail.com}

\begin{abstract}
In this article, we provide partition-theoretic interpretations for some new truncated pentagonal number theorem and identities of Gauss. Also, we deduce few inequalities for some partition functions.

\end{abstract}


\maketitle

\section{Introduction}\label{Int}
In 2012, Andrews and Merca \cite{AM1} proved truncated pentagonal number theorem
\begin{equation*}\label{APT}
\frac{1}{(q;q)_\infty}\sum_{\tau =0}^{\nu-1}(-1)^\tau q^{\tau (3\tau +1)/2}\left(1-q^{2\tau +1}\right)=1+(-1)^{\nu-1}\sum_{n=1}^{\infty}\dfrac{q^{\binom{\nu }{2}+(\nu +1)n}}{(q;q)_n}\left[\begin{array}{c}n-1\\ \nu -1\end{array}\right],
\end{equation*}
where \[(a;q)_\infty=\prod_{n=0}^{\infty}(1-aq^n),\qquad(a;q)_n=\frac{(a;q)_\infty}{(aq^n;q)_\infty}\]
and
\[\left[\begin{array}{c}M\\ N\end{array}\right]=\left[\begin{array}{c}M\\ N\end{array}\right]_q=\begin{cases}
0,\qquad \qquad\qquad \quad \,\text{if} \,N<0 \, \,\text{or}\, \, N>M,\\
\dfrac{(q;q)_M}{(q;q)_N(q;q)_{M-N}}, \,\text{otherwise}.
\end{cases}\]
As a consequence of the above result they obtained family of inequalities for the partition function $p(n)$ \cite{A}. Namely, for each $\nu \geq1$,
\begin{equation}\label{pe}
(-1)^{\nu -1}\sum_{\tau =0}^{\nu -1}(-1)^\tau \big(p\left(n-\tau (3\tau +1)/2\right)-p\left(n-\tau (3\tau +5)/2-1\right)\big)\geq 0
\end{equation}
with strict inequality if $n\geq \nu (3\nu +1)/2$. They also gave the partition interpretation for the truncated sum in \eqref{pe} as
\begin{equation*}\label{pe1}
(-1)^{\nu -1}\sum_{\tau =0}^{\nu -1}(-1)^\tau \big(p\left(n-\tau (3\tau +1)/2\right)-p\left(n-\tau (3\tau +5)/2-1\right)\big)=M_\nu (n),
\end{equation*}
where $M_\nu (n)$ counts the number of partitions of $n$ in which $\nu $ is the smallest integer that is not a part and there are more parts greater than $\nu $ compared to less than $\nu $.

Inspired by the above work, Guo and Zeng \cite{GZ} considered two idenitities of Gauss
\begin{align}\label{phi}
1+2\sum_{\tau =1}^{\infty}(-1)^\tau q^{\tau ^2}=\dfrac{(q;q)_\infty}{(-q;q)_\infty}=\phi(-q)
\end{align}
and
\begin{align}\label{psi}
\sum_{\tau =0}^{\infty}(-1)^\tau q^{\tau (2\tau +1)}\left(1-q^{2\tau +1}\right)=\dfrac{(q^2;q^2)_\infty}{(-q;q^2)_\infty}=\psi(-q)
\end{align}
to derive the truncated identities
\begin{align}\label{G1}
\dfrac{(-q;q)_\infty}{(q;q)_\infty}\left(1+2\sum_{\tau =1}^{\nu }(-1)^\tau q^{\tau ^2}\right)=1+(-1)^\nu \sum_{n=\nu +1}^{\infty}\dfrac{(-q;q)_\nu (-1;q)_{n-\nu }q^{(\nu +1)n}}{(q;q)_n}\left[\begin{array}{c}n-1\\ \nu \end{array}\right]
\end{align}
and
\begin{align}\label{G2}
\nonumber\dfrac{(-q;q^2)_\infty}{(q^2;q^2)_\infty}&\sum_{\tau =0}^{\nu -1}(-1)^\tau q^{\tau (2\tau +1)}\left(1-q^{2\tau +1}\right)\\& =1+(-1)^{\nu -1}\sum_{n=\nu }^{\infty}\dfrac{(-q;q^2)_\nu (-q;q^2)_{n-\nu }q^{2(\nu +1)n-\nu }}{(q^2;q^2)_n}\left[\begin{array}{c}n-1\\ \nu -1\end{array}\right]_{q^2},
\end{align}
for all $\nu \geq 1$.

Let $\overline{p}(n)$ be the number of overpartitions of $n$ (see Corteel and Lovejoy \cite{CL}) and $\textnormal{pod}(n)$ be the number of partitions of $n$ wherein odd parts are distinct (see Hirschhorn and Sellers \cite{HS}). The reciprocals of \eqref{phi} and \eqref{psi} give the generating functions for $\overline{p}(n)$ and $\textnormal{pod}(n)$, respectively.

Let $p_{2,4}(n)$ be the number of partitions of $n$ with parts $\not\equiv 2\pmod{4}$ and note that the generating function for $p_{2,4}(n)$ is same as that of $\textnormal{pod}(n)$.

From \eqref{G1} and \eqref{G2}, Guo and Zeng \cite{GZ} obtained the following inequalities
\begin{equation}\label{opi}
(-1)^\nu \left(\overline{p}(n)+2\sum_{\tau =1}^{\nu }(-1)^\tau \overline{p}(n-\tau ^2)\right)\geq 0
\end{equation}
with strict inequality if $n\geq(\nu +1)^2$ and
\begin{equation}\label{dopi}
(-1)^{\nu -1}\sum_{\tau =0}^{\nu -1}(-1)^\tau \big(\textnormal{pod}\left(n-\tau (2\tau +1)\right)-\textnormal{pod}\left(n-(\tau +1)(2\tau +1)\right)\big)\geq 0
\end{equation}
with strict inequality if $n\geq \nu (2\nu +1)$.

Andrews and Merca \cite{AM2} revised \eqref{G1} and \eqref{G2} as following:
\begin{equation*}\label{G3}
\dfrac{(-q;q)_\infty}{(q;q)_\infty}\left(1+2\sum_{\tau =1}^{\nu }(-1)^\tau q^{\tau ^2}\right)=1+2(-1)^\nu \dfrac{(-q;q)_\nu }{(q;q)_\nu }\sum_{\tau =0}^{\infty}\dfrac{q^{(\nu +1)(\nu +\tau +1)}(-q^{\nu +\tau +2};q)_\infty}{(1-q^{\nu +\tau +1})(q^{\nu +\tau +2};q)_\infty}
\end{equation*}
and
\begin{equation*}\label{G4}
\dfrac{(-q;q^2)_\infty}{(q^2;q^2)_\infty}\sum_{\tau =0}^{2\nu -1}(-q)^{\tau (\tau +1)/2} =1-(-1)^\nu \dfrac{(-q;q^2)_\nu }{(q^2;q^2)_{\nu -1}}\sum_{\tau =0}^{\infty}\dfrac{q^{\nu (2\tau +2\nu +1)}(-q^{2\tau +2\nu +3};q^2)_\infty}{(q^{2\nu +2\tau +2};q^2)_\infty}.
\end{equation*}
From these two identities they deduced interpretations of the sums in the inequalities \eqref{opi} and \eqref{dopi}:
\begin{equation*}\label{opp}
(-1)^\nu \left(\overline{p}(n)+2\sum_{\tau =1}^{\nu }(-1)^\tau \overline{p}(n-\tau ^2)\right)=\overline{\mu}_\nu (n),
\end{equation*}
where $\overline{\mu}_\nu (n)$ counts the number of overpartitions of $n$ in which the first part greater than $\nu $ appears at least $\nu +1$ times and
\begin{equation*}\label{dopp}
(-1)^{\nu -1}\sum_{\tau =0}^{\nu -1}(-1)^\tau \big(\textnormal{pod}\left(n-\tau (2\tau +1)\right)-\textnormal{pod}\left(n-(\tau +1)(2\tau +1)\right)\big)=MP_\nu (n),
\end{equation*}
where $MP_\nu (n)$ is the number of partitions of $n$ in which the first part greater than $2\nu -1$ is odd and appears exactly $\nu $ times. All other odd parts appear at most once.

To know more about truncated identities one can read Burnette and Kolitsch \cite{BK}, Chan et al. \cite{CH}, He et al. \cite{HJZ}, Kolitsch \cite{K}, Mao \cite{M}, Merca et al. \cite{MWY}, Wang and Yee \cite{WY,WY1}, Yee \cite{Y}.

Furthermore Andrews and Merca \cite{AM2} conjectured that
\begin{conj}
	For $\nu $ even or $n$ odd,
	\begin{align}\label{co1}
	(-1)^{\nu -1}\sum_{i=0}^{\nu -1}(-1)^i\big(p\left(n-i(2i+1)\right)-p\left(n-(i+1)(2i+1)\right)\big)\leq M_\nu (n).
	\end{align}
\end{conj}

We establish a partition-theoretic interpretation for the truncated sum in the inequality \eqref{co1} and we show that the sum is bounded above by $p(n-\nu (2\nu +1))$.

In this article, we obtain several  partition-theoretic interpretaion of some new truncated sums of pentagonal number theorem and two identities of Gauss. Also, we deduce inequalities for various partition functions.

MacMahon \cite[p.13]{A} introduced what he termed modular partitions. Given positive integers $m$ and $n$, there exist $b\geq 0$ and $1\leq r\leq m$ such that $n=mb+r$. The $m$-modular partitions are a modification of the Young diagram so that $n$ is represented by a row of $b$ boxes with $m$ in each of them and one box with $r$ in it. For convenience, we put $r$ in the first column of the Young diagram. For example, the below figure shows the $3$-modular Young diagram of the partition $16+14+7+5+2$.
\begin{figure*}[!h]
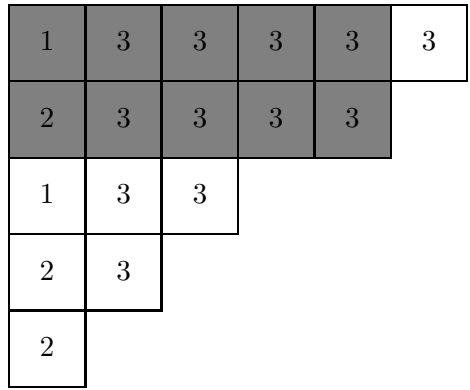

	\centering
	\ytableausetup{centertableaux,mathmode, boxsize=1cm,mathmode}
	\begin{ytableau}
		*(gray) 1 & *(gray)3 & *(gray)3 & *(gray)3 & *(gray)3 & 3 \\
		*(gray)2 & *(gray)3 & *(gray)3 & *(gray)3 & *(gray)3 \\
		1 & 3 & 3 \\
		2 & 3 \\
		2 \\
	\end{ytableau}
	\caption{The $3$-modular Young diagram of 16+14+7+5+2}
\end{figure*}

For a Young diagram of partition $\pi$, define the $\nu $-Durfee rectangle, $\nu $ being a nonnegative integer, to be the largest rectangle which fits in the graph whose width minus its height is $\nu $. In Fig. 1, the $3$-Durfee rectangle of the partition is the shaded rectangle of size $2\times 5$.

For a fixed $\nu >0$ and $n\geq 0$, define $M(a,m,\nu ;n)$ to be the number of partitions of $n$ into the parts $\equiv a\pmod{m}$ such that all parts $\leq m\nu +a$ occur as a part at least once and the parts below the $(\nu +2)$-Durfee rectangle in the $m$-modular graph are strictly less than the width of the rectangle. For example, the partition $20+17+11+8+5+2$ is counted by $M(2,3,2;63)$. Whereas, the partition $23+17+17+8+5+5+2$ is not counted by $M(2,3,2;77)$.

\begin{theorem}\label{M3}
	For a fixed $\nu \geq 0$ and positive integers $a$, $m$ such that $a<m$, we have
	\begin{equation}\label{M13}
	\sum_{n=0}^{\infty}M(a,m,\nu ;n)q^n=q^{a+\nu (m\nu +m+2a)/2}\sum_{\tau =0}^{\infty}\frac{q^{\tau (m\nu +m\tau +m+a)}}{(q^m;q^m)_\tau (q^a;q^m)_{\nu +\tau +1}}
	\end{equation}	
\end{theorem}

For a fixed $\nu >0$ and $n\geq 0$, define $N_\nu (n)$ be the number of partitions of $n$ such that all parts $\leq \nu -1$ occur at least once and the parts below the $\nu $-Durfee rectangle in the Young diagram are strictly less than the width of the rectangle. For example, the partition $6+5+3+3+2+1$ is counted by $N_3(20)$. Whereas, the partition $6+5+5+2+1+1$ is not counted by $N_3(20)$.
\begin{theorem}\label{Nk}
	For a fixed $\nu \geq 1$, we have
	\begin{equation}\label{nk1}
	\sum_{n=0}^{\infty}N_\nu (n)q^n=q^{\nu (\nu -1)/2}\sum_{\tau =0}^{\infty}\frac{q^{\tau (\nu +\tau )}}{(q;q)_\tau (q;q)_{\nu +\tau -1}}.
	\end{equation}
\end{theorem}

For a fixed $\nu \geq 0$ and $n\geq 0$, define $p(a,m,\nu ;n)$ to be the number of partitions of $n$ in which parts $\equiv a\pmod{m}$ form a partition counted by $M(a,m,\nu ;n-\beta)$, $\beta$ is the sum of parts $\not\equiv a\pmod{m}$.

\begin{theorem}\label{T2}
	For $\nu \geq1$ and $n> \nu (3\nu +5)/2$,
	\begin{align}\label{pi}
	\nonumber(-1)^\nu &\left(\sum_{\tau =0}^{\nu }(-1)^\tau p\left(n-\tau (3\tau +1)/2\right)-\sum_{\tau =0}^{\nu -1}(-1)^\tau p\left(n-\tau (3\tau +5)/2-1\right)\right)\\&=p(2,3,\nu ;n)+p(1,3,\nu ;n).
	\end{align}
\end{theorem}

For example, let $\nu =2$ and $n=17$.

\vspace{0.4cm}

\begin{tabular}{ l|l } 
	\hline
	Partitions counted  & Partitions counted \\ by  $p(2,3,2;17)$ & by $p(1,3,2;17)$\\ 
	\hline
	$8+5+2+2$ & $7+5+4+1$ \\ 
	$8+5+2+1+1$ & $7+4+4+1+1$ \\ 
	&$7+4+3+2+1$\\
	&$7+4+3+1+1+1$\\
	&$7+4+2+2+1+1$\\
	&$7+4+2+1+1+1+1$\\
	&$7+4+1+1+1+1+1+1$\\
\end{tabular}
\vspace{0.4cm}

Hence, $p(17)-p(15)+p(10)-p(16)+p(12)=9=p(2,3,2;17)+p(1,3,2;17).$

\begin{theorem}\label{T3}
	For $\nu \geq 1$ and $n>\nu (2\nu +3)$,
	\begin{align}\label{p1i}
	\nonumber(-1)^\nu &\left(\sum_{\tau =0}^{\nu }(-1)^\tau p\left(2n-\tau (2\tau +1)\right)-\sum_{\tau =0}^{\nu -1}(-1)^\tau p\left(2n-(\tau +1)(2\tau +1)\right)\right)\\&=(-1)^\nu p_o(n)+p(3,4,\nu ;2n)+p(1,4,\nu ;2n)
	\end{align}
	and
	\begin{align}\label{p2i}
	\nonumber(-1)^\nu &\left(\sum_{\tau =0}^{\nu }(-1)^\tau p\left(2n+1-\tau (2\tau +1)\right)-\sum_{\tau =0}^{\nu -1}(-1)^\tau p\left(2n+1-(\tau +1)(2\tau +1)\right)\right)\\&=p(3,4,\nu ;2n+1)+p(1,4,\nu ;2n+1).
	\end{align}
\end{theorem}

\begin{theorem}\label{T4}
	For $\nu \geq 1$ and $n> \nu (2\nu +3)$,
	\begin{align}\label{dpi}
	\nonumber(-1)^\nu &\left(\sum_{\tau =0}^{\nu }(-1)^\tau p_o\left(2n+1-\tau (2\tau +1)\right)-\sum_{\tau =0}^{\nu -1}(-1)^\tau p_o\left(2n+1-(\tau +1)(2\tau +1)\right)\right)\\&= p_o(3,4,\nu ;2n+1)+p_o(1,4,\nu ;2n+1),
	\end{align}
	where $p_o(a,4,\nu ;n)$ $(a\in\{1,3\})$, counts partitions of $n$ into odd parts in which the parts $\equiv a\pmod{4}$ form a partition counted by $M(a,4,\nu ;n-\alpha)$, $\alpha$ is the sum of parts $\not\equiv a\pmod{4}$.
\end{theorem}

For a fixed $\nu \geq 0$ and $n\geq 0$, $p_{2,4}(a,4,\nu ;n)$ counts partitions of $n$ into the parts $\not\equiv 2\pmod{4}$ in which the parts $\equiv a\pmod{4}$  form a partition counted by $M(a,4,\nu ;n-\alpha)$, $\alpha$ is the sum of parts $\not\equiv a\pmod{4}$.

\begin{theorem}\label{T5}
	For $\nu \geq 1$ and $n> \nu (2\nu +3)$,
	\begin{align}\label{podi}
	\nonumber	(-1)^\nu &\left(\sum_{\tau =0}^{\nu }(-1)^\tau p_{2,4}\left(n-\tau (2\tau +1)\right)-\sum_{\tau =0}^{\nu -1}(-1)^\tau p_{2,4}\left(n-(\tau +1)(2\tau +1)\right)\right)\\&=p_{2,4}(3,4,\nu ;n)+p_{2,4}(1,4,\nu ;n).
	\end{align}
	
\end{theorem}

\begin{theorem}\label{T6}
	For $\nu \geq 1$ and $n\geq (\nu +1)^2$,
	\begin{equation}\label{oppi}
	(-1)^\nu \left(\overline{p}(n)+2\sum_{\tau =1}^{\nu }(-1)^\tau \overline{p}(n-\tau ^2)\right)=2 \overline{p}(1,2,\nu ;n),
	\end{equation}
	where $\overline{p}(1,2,\nu ;n)$ counts overpartitions of $n$ in which the non-overlined odd parts form a partition counted by $M(1,2,\nu ;n-\alpha)$, $\alpha$ is the sum of parts other than non-overlined odd parts.
\end{theorem}

\begin{theorem}\label{T7}
	For $\nu \geq 1$ and $n\geq (\nu +1)^2$,
	\begin{equation}\label{ppi}
	(-1)^\nu \left(pp(2n)+2\sum_{\tau =1}^{\nu }(-1)^\tau  pp(2n-\tau ^2)\right)=(-1)^\nu p(n)+2pp_e(2n)
	\end{equation}
	and
	\begin{equation}\label{ppi1}
	(-1)^\nu \left(pp(2n+1)+2\sum_{\tau =1}^{\nu }(-1)^\tau  pp(2n+1-\tau ^2)\right)=2pp_e(2n+1),
	\end{equation}
	where $pp_e(n)$ counts bipartitions $(\pi_1,\pi_2)$ such that $\pi_1$ is a unrestricted partition and $\pi_2$ is a partition in which odd parts are counted by $M(1,2,\nu ;n-|\pi_1|-\alpha)$, $|\pi_1|$ is the sum of parts of $\pi_1$ and $\alpha$ is the sum of even parts of $\pi_2$.
\end{theorem}
\begin{theorem}\label{T8}
For each $n\geq 1$,
\begin{equation}\label{eq8}
p(n)=p(2,3,0;n)+p(1,3,0;n)
\end{equation}
and
\begin{equation}\label{eq9}
p_{2,4}(n)=p_{2,4}(3,4,0;n)+p_{2,4}(1,4,0;n).
\end{equation}

\end{theorem}
\begin{theorem}\label{CP}
	For $\nu \geq 1$ and $n\geq 0$,
	\begin{equation}\label{CP1}
	\sum_{\tau =0}^{\infty}(-1)^\tau p\left(n-\dfrac{(\nu +\tau )(\nu +\tau -1)}{2}\right)=N_\nu (n).
	\end{equation}
\end{theorem}

For example, let $\nu =2$ and $n=10$. Partitions counted by  $N_2(10)$ are $1+1+1+1+1+1+1+1+1+1$, $9+1$, $8+1+1$, $7+2+1$, $7+1+1+1$, $6+2+1+1$, $6+1+1+1+1$, $5+2+2+1$, $5+2+1+1+1$, $5+1+1+1+1+1$, $4+2+2+1+1$, $4+2+1+1+1+1$, $4+1+1+1+1+1+1$, $3+2+2+2+1$, $3+2+2+1+1+1$, $3+2+1+1+1+1+1$, $3+1+1+1+1+1+1+1$, $5+4+1$, $4+4+1+1$. 
Hence, $p(9)-p(7)+p(4)-p(0)=19=N_2(10).$

The following inequality is obvious from the last theorem.
\begin{corollary}\label{CI}
	For $\nu \geq 1$ and $n\geq 0$,
	\begin{equation}
	\sum_{\tau =0}^{\infty}(-1)^\tau p\left(n-\dfrac{(\nu +\tau )(\nu +\tau -1)}{2}\right)\geq 0,
	\end{equation}
	here we assume $p(x)=0$ if $x$ is a negative integer.
\end{corollary}
To prove our theorems, we rely on Gauss hypergeometric series \cite[Eqn. 1.2.14]{GR}
\begin{equation*}\label{GH}
_2\phi_1\left(\genfrac{}{}{0pt}{}{\alpha,\beta}{\gamma};q,z\right)=\sum_{n=0}^{\infty}\frac{(\alpha;q)_n(\beta;q)_n}{(q;q)_n(\gamma;q)_n}z^n
\end{equation*}
and second identity by Heine's transformation of $_2 \phi_1$ series \cite[Eqn. 1.4.5]{GR},
\begin{equation*}\label{HT}
_2\phi_1\left(\genfrac{}{}{0pt}{}{\alpha,\beta}{\gamma};q,z\right)=\dfrac{(\gamma/\beta;q)_\infty(\beta z;q)_\infty}{(\gamma;q)_\infty(z;q)_\infty}\
_2\phi_1\left(\genfrac{}{}{0pt}{}{\alpha \beta z/\gamma,\beta}{\beta z};q,\gamma/\beta\right).
\end{equation*}
Without explicitly mentioning we frequently use these two identities in our proofs.

\section{Proofs of main theorems}\label{MR}
In this section, we prove Theorems \ref{M3}--\ref{CP}.
\subsection*{Proof of Theorem \ref{M3}}
For a fixed $\nu \geq 0$, the partitions counted by $M(a,m,\nu ;n)$ with $(\nu +2)$-Durfee rectangle, in a $3$-modular Young diagram, of size $\tau \times (\nu +\tau +2) $ are composed of three pieces. The first piece is Durfee rectangle with enumerator $q^{\tau (m\nu +m\tau +m+a)}$, another is the piece right to the Durfee rectangle which are partitions with $\leq \tau$ parts and the last piece under the Durfee rectangle are the partitions into parts $\leq \nu +\tau+1$. It is clear from the definition of $M(a,m,\nu ;n)$ that the parts below and the parts right to the Durfee rectangle are generated by $q^{a+\nu (m\nu +m+2a)/2}/(q^a;q^m)_{\nu +\tau +1}$ and $1/(q^m;q^m)_\tau $, respectively. Here, $q^{a+\nu (m\nu +m+2a)/2}$ accounts for all parts $\equiv a\pmod{m}$ which are $\leq m\nu +a$.

\subsection*{Proof of Theorem \ref{Nk}}
Proof of Theorem \ref{Nk} is similar to that of \ref{M3}, so we omit the details.

\subsection*{Proof of Theorem \ref{T2}}
From the pentagonal number theorem, we have
\begin{align}\label{eqt1}
\nonumber(q;q)_\infty&=\sum_{\tau =0}^{\infty}(-1)^\tau q^{\tau (3\tau +1)/2}+\sum_{\tau =-1}^{-\infty}(-1)^\tau q^{\tau (3\tau +1)/2}\\&
=\sum_{\tau =0}^{\infty}(-1)^\tau q^{\tau (3\tau +1)/2}-\sum_{\tau =0}^{\infty}(-1)^\tau q^{\tau (3\tau +5)/2+1}.
\end{align}
Thus,
\begin{align}\label{eq1}
\nonumber&\dfrac{1}{(q;q)_\infty}\left(\sum_{\tau =0}^{\nu }(-1)^\tau q^{\tau (3\tau +1)/2}-\sum_{\tau =0}^{\nu -1}(-1)^\tau q^{\tau (3\tau +5)/2+1}\right)\\&
\nonumber=1-\dfrac{1}{(q;q)_\infty}\sum_{\tau =\nu +1}^{\infty}(-1)^\tau q^{\tau (3\tau +1)/2}+\dfrac{1}{(q;q)_\infty}\sum_{\tau =\nu }^{\infty}(-1)^\tau q^{\tau (3\tau +5)/2+1}\\&\nonumber
=1+\dfrac{(-1)^\nu q^{2+\nu (3\nu +7)/2}}{(q;q)_\infty}\sum_{\tau =0}^{\infty}(-1)^\tau q^{\tau (3\tau +7)/2+3\nu \tau }+\dfrac{(-1)^\nu q^{1+\nu (3\nu +5)/2}}{(q;q)_\infty}\sum_{\tau =0}^{\infty}(-1)^\tau q^{\tau (3\tau +5)/2+3\nu \tau }\\&
\nonumber=1+\dfrac{(-1)^\nu q^{2+\nu (3\nu +7)/2}}{(q;q)_\infty}\lim_{\delta\rightarrow 0}\ _2\phi_1\left(\genfrac{}{}{0pt}{}{q^3,\dfrac{q^{3\nu +5}}{\delta}}{0};q^3,\delta\right)\\&\quad
\nonumber+\dfrac{(-1)^\nu q^{1+\nu (3\nu +5)/2}}{(q;q)_\infty}\lim_{\delta\rightarrow 0} \ _ 2\phi_1\left(\genfrac{}{}{0pt}{}{q^3,\dfrac{q^{3\nu +4}}{\delta}}{0};q^3,\delta\right)\\&
\nonumber=1+\dfrac{(-1)^\nu q^{2+\nu (3\nu +7)/2}}{(q;q)_\infty}\lim_{\delta\rightarrow 0}\dfrac{(q^{3\nu +5};q^3)_\infty}{(\delta;q^3)_\infty}\sum_{\tau =0}^{\infty}\dfrac{(-1)^\tau \delta^\tau q^{3\tau (\tau +1)/2}\left(q^{3\nu +5}/\delta;q^3\right)_\tau }{(q^3;q^3)_\tau (q^{3\nu +5};q^3)_\tau }\\&\quad
\nonumber+\dfrac{(-1)^\nu q^{1+\nu (3\nu +5)/2}}{(q;q)_\infty}\lim_{\delta\rightarrow 0}\dfrac{(q^{3\nu +4};q^3)_\infty}{(\delta;q^3)_\infty}\sum_{\tau =0}^{\infty}\dfrac{(-1)^\tau \delta^\tau q^{3\tau (\tau +1)/2}\left(q^{3\nu +4}/\delta;q^3\right)_\tau }{(q^3;q^3)_\tau (q^{3\nu +4};q^3)_\tau }
\\&
\nonumber=1+(-1)^\nu q^{2+\nu (3\nu +7)/2}\dfrac{(q^{3\nu +5};q^3)_\infty}{(q;q)_\infty}\sum_{\tau =0}^{\infty}\dfrac{q^{\tau (3\nu +3\tau +5)}}{(q^3;q^3)_\tau \left(q^{3\nu +5};q^3\right)_\tau }\\&\quad
\nonumber+(-1)^\nu q^{1+\nu (3\nu +5)/2}\dfrac{(q^{3\nu +4};q^3)_\infty}{(q;q)_\infty}\sum_{\tau =0}^{\infty}\dfrac{q^{\tau (3\nu +3\tau +4)}}{(q^3;q^3)_\tau \left(q^{3\nu +4};q^3\right)_\tau }\\&
\nonumber=1+(-1)^\nu q^{2+\nu (3\nu +7)/2}\frac{(q^2;q^3)_\infty}{(q;q)_\infty}\sum_{\tau =0}^{\infty}\frac{q^{\tau (3\nu +3\tau +5)}}{(q^3;q^3)_\tau (q^2;q^3)_{\nu +\tau +1}}\\&\quad
+(-1)^\nu q^{1+\nu (3\nu +5)/2}\frac{(q;q^3)_\infty}{(q;q)_\infty}\sum_{\tau =0}^{\infty}\frac{q^{\tau (3\nu +3\tau +4)}}{(q^3;q^3)_\tau (q;q^3)_{\nu +\tau +1}}.
\end{align}
Using Theorem \ref{M3} and equation \ref{eq1}, we arrive at \eqref{pi}.

\subsection*{Proof of Theorem \ref{T3}}\label{OP}
From the Gauss second identity \eqref{psi}, we have
\begin{align}
\nonumber&\dfrac{1}{\psi(-q)}\left(\sum_{\tau =0}^{\nu }(-1)^\tau q^{\tau (2\tau +1)}-\sum_{\tau =0}^{\nu -1}(-1)^\tau q^{(\tau +1)(2\tau +1)}\right)\\&\nonumber=1-\dfrac{1}{\psi(-q)}\sum_{\tau =\nu +1}^{\infty}(-1)^\tau q^{\tau (2\tau +1)}+\dfrac{1}{\psi(-q)}\sum_{\tau =\nu }^{\infty}(-1)^\tau q^{(\tau +1)(2\tau +1)}
\\&\nonumber
=1+(-1)^\nu \dfrac{q^{3+\nu (2\nu +5)}}{\psi(-q)}\sum_{\tau =0}^{\infty}(-1)^\tau q^{\tau (4\nu +2\tau +5)}+(-1)^\nu \dfrac{q^{1+\nu (2\nu +3)}}{\psi(-q)}\sum_{\tau =0}^{\infty}(-1)^\tau q^{\tau (4\nu +2\tau +3)}\\&\nonumber
=1+(-1)^\nu \dfrac{q^{3+\nu (2\nu +5)}}{\psi(-q)}\lim_{\delta\rightarrow 0}\ _2\phi_1\left(\genfrac{}{}{0pt}{}{q^4,\dfrac{q^{4\nu +7}}{\delta}}{0};q^4,\delta\right)+(-1)^\nu \dfrac{q^{1+\nu (2\nu +3)}}{\psi(-q)}\lim_{\delta\rightarrow 0}\ _2\phi_1\left(\genfrac{}{}{0pt}{}{q^4,\dfrac{q^{4\nu +5}}{\delta}}{0};q^4,\delta\right)
\\&\nonumber
=1+(-1)^\nu \dfrac{q^{3+\nu (2\nu +5)}}{\psi(-q)}\lim_{\delta\rightarrow 0}\dfrac{(q^{4\nu +7};q^4)_\infty}{(\delta;q^4)_\infty}\sum_{\tau =0}^{\infty}\dfrac{(-1)^\tau \delta^\tau q^{2\tau (\tau +1)}(q^{4\nu +7}/\delta;q^4)_\tau }{(q^4;q^4)_\tau (q^{4\nu +7};q)_\tau }\\&\nonumber\quad
+(-1)^\nu \dfrac{q^{1+\nu (2\nu +3)}}{\psi(-q)}\lim_{\delta\rightarrow 0}\dfrac{(q^{4\nu +5};q^4)_\infty}{(\delta;q^4)_\infty}\sum_{\tau =0}^{\infty}\dfrac{(-1)^\tau \delta^\tau q^{2\tau (\tau +1)}(q^{4\nu +5}/\delta;q^4)_\tau }{(q^4;q^4)_\tau (q^{4\nu +5};q)_\tau }\\&\nonumber
=1+(-1)^\nu q^{3+\nu (2\nu +5)}\frac{(q^3;q^4)_\infty}{\psi(-q)}\sum_{\tau =0}^{\infty}\frac{q^{\tau (4\nu +4\tau +7)}}{(q^4;q^4)_\tau (q^3;q^4)_{\nu +\tau +1}}\\&\quad
+(-1)^\nu q^{1+\nu (2\nu +3)}\frac{(q;q^4)_\infty}{\psi(-q)}\sum_{\tau =0}^{\infty}\frac{q^{\tau (4\nu +4\tau +5)}}{(q^4;q^4)_\tau (q;q^4)_{\nu +\tau +1}}.\label{eq2}
\end{align}
Multiplying equation \eqref{eq2} by $(q^4;q^4)_\infty/(q^2;q^2)_\infty$ and using $\psi(-q)=\frac{(q;q)_\infty(q^4;q^4)_\infty}{(q^2;q^2)_\infty}$, we obtain
\begin{align*}
\dfrac{1}{(q;q)_\infty}&\left(\sum_{\tau =0}^{\nu }(-1)^\tau q^{\tau (2\tau +1)}-\sum_{\tau =0}^{\nu -1}(-1)^\tau q^{(\tau +1)(2\tau +1)}\right)\\&\nonumber=\dfrac{(q^4;q^4)_\infty}{(q^2;q^2)_\infty}+
(-1)^\nu q^{3+\nu (2\nu +5)}\frac{(q^3;q^4)_\infty}{(q;q)_\infty}\sum_{\tau =0}^{\infty}\frac{q^{\tau (4\nu +4\tau +7)}}{(q^4;q^4)_\tau (q^3;q^4)_{\nu +\tau +1}}\\&\quad+(-1)^\nu q^{1+\nu (2\nu +3)}\frac{(q;q^4)_\infty}{(q;q)_\infty}\sum_{\tau =0}^{\infty}\frac{q^{\tau (4\nu +4\tau +5)}}{(q^4;q^4)_\tau (q;q^4)_{\nu +\tau +1}},
\end{align*}
which implies that
\begin{align}\label{dop2}
\nonumber(-1)^\nu &\sum_{n=0}^{\infty}p(n)q^n\left(\sum_{\tau =0}^{\nu }(-1)^\tau q^{\tau (2\tau +1)}-\sum_{\tau =0}^{\nu -1}(-1)^\tau q^{(\tau +1)(2\tau +1)}\right)\\&
=(-1)^\nu \sum_{n=0}^{\infty}p_o(n)q^{2n}+\sum_{n=0}^{\infty}p(3,4,\nu ;n)q^n+\sum_{n=0}^{\infty}p(1,4,\nu ;n)q^n.
\end{align}
Equating coefficients of $q^{2n}$ and $q^{2n+1}$ on both sides of \eqref{dop2}, we arrive at \eqref{p1i} and \eqref{p2i}, respectively.

\subsection*{Proof of Theorem \ref{T4}}\label{DOP}
Proof of \eqref{dpi} is similar to that of \eqref{p2i} but instead of multiplying $(q^4;q^4)_\infty/(q^2;q^2)_\infty$, we multiply $(q^4;q^4)_\infty$.
\subsection*{Proof of Theorem \ref{T5}}
Equation \eqref{podi} follows from \eqref{eq2} and Theorem \ref{M3}.
\subsection*{Proof of Theorem \ref{T6}}
From the Gauss identity \eqref{phi}, we have
\begin{align}
\nonumber\dfrac{1}{\phi(-q)}\left(1+2\sum_{\tau =1}^{\nu }(-1)^\tau q^{\tau ^2}\right)&=1-\dfrac{2}{\phi(-q)}\sum_{\tau =\nu +1}^{\infty}(-1)^\tau q^{\tau ^2}
\\&\nonumber=1+\dfrac{2(-1)^\nu q^{(\nu +1)^2}}{\phi(-q)}\sum_{\tau =0}^{\infty}(-1)^\tau q^{\tau ^2+2\tau (\nu +1)}\\&
\nonumber=1+\dfrac{2(-1)^\nu q^{(\nu +1)^2}}{\phi(-q)}\lim_{\delta\rightarrow 0}\ _2\phi_1\left(\genfrac{}{}{0pt}{}{q^2,\dfrac{q^{2\nu +3}}{\delta}}{0};q^2,\delta\right)\\&\nonumber
=1+\dfrac{2(-1)^\nu q^{(\nu +1)^2}}{\phi(-q)}\lim_{\delta\rightarrow 0}\dfrac{(q^{2\nu +3};q^2)_\infty}{(\delta;q)_\infty}\sum_{\tau =0}^{\infty}\dfrac{(-1)^\tau \delta^\tau q^{\tau (\tau +1)}(q^{2\nu +3}/\delta;q^2)_\tau }{(q^2;q^2)_\tau (q^{2\nu +3};q^2)_\tau }\\&\nonumber
=1+\dfrac{2(-1)^\nu q^{(\nu +1)^2}}{\phi(-q)}(q^{2\nu +3};q^2)_\infty\sum_{\tau =0}^{\infty}\dfrac{q^{2\tau ^2+(2\nu +3)\tau }}{(q^2;q^2)_\tau (q^{2\nu +3};q^2)_\tau }\\&\nonumber
=1+\dfrac{2(-1)^\nu q^{(\nu +1)^2}}{(q;q)_\infty(q;q^2)_\infty}(q^{2\nu +3};q^2)_\infty\sum_{\tau =0}^{\infty}\dfrac{q^{\tau (2\nu +2\tau +3)}}{(q^2;q^2)_\tau (q^{2\nu +3};q^2)_\tau }\\&
=1+\dfrac{2(-1)^\nu q^{(\nu +1)^2}}{(q;q)_\infty}\sum_{\tau =0}^{\infty}\dfrac{q^{\tau (2\nu +2\tau +3)}}{(q^2;q^2)_\tau (q;q^2)_{\nu +\tau +1}}.\label{eq6}
\end{align}
Equation \eqref{oppi} follows from \ref{eq6} and Theorem \ref{M3}.
\subsection*{Proof of Theorem \ref{T7}}
Multiplying the equation \eqref{eq6} by $1/(q^2;q^2)_\infty$ and using $\phi(-q)=(q;q)^2_\infty/(q^2;q^2)_\infty$, we obtain
\begin{equation*}\label{op1}
\dfrac{1}{(q;q)^2_\infty}\left(1+2\sum_{\tau =1}^{\nu }(-1)^\tau q^{\tau ^2}\right)=\dfrac{1}{(q^2;q^2)_\infty}+2(-1)^\nu \frac{q^{(\nu +1)^2}}{(q;q)_\infty(q^2;q^2)_\infty}\sum_{\tau =0}^{\infty}\frac{q^{\tau (2\nu +2\tau +3)}}{(q^2;q^2)_\tau (q;q^2)_{\nu +\tau +1}}
\end{equation*}
and the statements follow easily from Theorem \ref{M3} and the above equation.
\subsection*{Proof of Theorem \ref{T8}}
From the equation \eqref{eqt1}, we have
\begin{align}
\nonumber&\dfrac{1}{(q;q)_\infty}=1-\dfrac{1}{(q;q)_\infty}\sum_{\tau =1}^{\infty}(-1)^\tau q^{\tau (3\tau +1)/2}+\dfrac{1}{(q;q)_\infty}\sum_{\tau =0}^{\infty}(-1)^\tau q^{\tau (3\tau +5)/2+1}\\&\nonumber
=1+\dfrac{q^2}{(q;q)_\infty}\sum_{\tau =0}^{\infty}(-1)^\tau q^{\tau (3\tau +7)/2}+\dfrac{q}{(q;q)_\infty}\sum_{\tau =0}^{\infty}(-1)^\tau q^{\tau (3\tau +5)/2}\\&
\nonumber=1+\dfrac{q^2}{(q;q)_\infty}\lim_{\delta\rightarrow 0}\ _2\phi_1\left(\genfrac{}{}{0pt}{}{q^3,\dfrac{q^5}{\delta}}{0};q^3,\delta\right)+\dfrac{q}{(q;q)_\infty}\lim_{\delta\rightarrow 0} \ _ 2\phi_1\left(\genfrac{}{}{0pt}{}{q^3,\dfrac{q^4}{\delta}}{0};q^3,\delta\right)\\&
\nonumber=1+\dfrac{q^2}{(q;q)_\infty}\lim_{\delta\rightarrow 0}\dfrac{(q^5;q^3)_\infty}{(\delta;q^3)_\infty}\sum_{\tau =0}^{\infty}\dfrac{(-1)^\tau \delta^\tau q^{3\tau (\tau +1)/2}\left(q^5/\delta;q^3\right)_\tau }{(q^3;q^3)_\tau (q^5;q^3)_\tau }\\&\quad
\nonumber+\dfrac{q}{(q;q)_\infty}\lim_{\delta\rightarrow 0}\dfrac{(q^4;q^3)_\infty}{(\delta;q^3)_\infty}\sum_{\tau =0}^{\infty}\dfrac{(-1)^\tau \delta^\tau q^{3\tau (\tau +1)/2}\left(q^4/\delta;q^3\right)_\tau }{(q^3;q^3)_\tau (q^4;q^3)_\tau }
\\&
\nonumber=1+q^{2}\dfrac{(q^{5};q^3)_\infty}{(q;q)_\infty}\sum_{\tau =0}^{\infty}\dfrac{q^{\tau (3\tau +5)}}{(q^3;q^3)_\tau \left(q^{5};q^3\right)_\tau }+q\dfrac{(q^{4};q^3)_\infty}{(q;q)_\infty}\sum_{\tau =0}^{\infty}\dfrac{q^{\tau (3\tau +4)}}{(q^3;q^3)_\tau \left(q^{4};q^3\right)_\tau }\\&
=1+q^{2}\frac{(q^2;q^3)_\infty}{(q;q)_\infty}\sum_{\tau =0}^{\infty}\frac{q^{\tau (3\tau +5)}}{(q^3;q^3)_\tau (q^2;q^3)_{\tau +1}}
+q\frac{(q;q^3)_\infty}{(q;q)_\infty}\sum_{\tau =0}^{\infty}\frac{q^{\tau (3\tau +4)}}{(q^3;q^3)_\tau (q;q^3)_{\tau +1}}.\label{eqt8}
\end{align}
Using equation \eqref{eqt8} and Theorem \ref{M3}, we arrive at \eqref{eq8}.

Similarly one can prove \eqref{eq9} by using Gauss second identity \eqref{psi}.
\subsection*{Proof of Theorem \ref{CP}} First, we will prove the following lemma:
\begin{lemma}
	For $\nu \geq 1$ and $m\geq 0$, we have
	\begin{equation}\label{l1}
	\sum_{n=0}^{\infty}N_\nu (n)q^n=\sum_{n=0}^{\infty}p(n)q^n\sum_{\tau =0}^{m}(-1)^\tau q^{(\nu +\tau )(\nu +\tau -1)/2}+(-1)^{m+1}\sum_{n=0}^{\infty}N_{\nu +m+1}(n)q^n.
	\end{equation}
\end{lemma}
\begin{proof}
From the Theorem \ref{Nk}, we have
	\begin{align}\label{p1}
	\nonumber\sum_{n=0}^{\infty}N_\nu (n)q^n&=q^{\nu (\nu -1)/2}\sum_{\tau =0}^{\infty}\frac{q^{\tau (\nu +\tau )}}{(q;q)_\tau (q;q)_{\nu +\tau -1}}\\&\nonumber
	=q^{\nu (\nu -1)/2}\sum_{\tau =0}^{\infty}\frac{(1-q^{\nu +\tau })q^{\tau (\nu +\tau )}}{(q;q)_\tau (q;q)_{\nu +\tau }}\\&\nonumber
	=q^{\nu (\nu -1)/2}\sum_{\tau =0}^{\infty}\frac{q^{\tau (\nu +\tau )}}{(q;q)_\tau (q;q)_{\nu +\tau }}-q^{\nu (\nu +1)/2}\sum_{\tau =0}^{\infty}\frac{q^{\tau (\nu +\tau +1)}}{(q;q)_\tau (q;q)_{\nu +\tau }}\\&\
	=q^{\nu (\nu -1)/2}\sum_{n=0}^{\infty}p(n)q^n-\sum_{n=0}^{\infty}N_{\nu +1}(n)q^n,\quad (\text{from \cite[p. 91]{GH} and \eqref{nk1}})
	\end{align}
	which is $m=0$ case of \eqref{l1}. Now assume that \eqref{l1} is holds for some $m\geq 0$, then, from \eqref{l1} and \eqref{p1}, we have
	\begin{align*}
	\sum_{n=0}^{\infty}N_\nu (n)q^n&=\sum_{n=0}^{\infty}p(n)q^n\sum_{\tau =0}^m(-1)^\tau q^{(\nu +\tau )(\nu +\tau -1)/2}\\&\quad+(-1)^{m+1}\left(q^{(\nu +m)(\nu +m+1)/2}\sum_{n=0}^{\infty}p(n)q^n-\sum_{n=0}^{\infty}N_{\nu +m+2 }(n)q^n\right)\\&
	=\sum_{n=0}^{\infty}p(n)q^n\sum_{\tau =0}^{m+1}(-1)^\tau q^{(\nu +\tau )(\nu +\tau -1)/2}+(-1)^{m+2}\sum_{n=0}^{\infty}N_{\nu +m+2}(n)q^n,
	\end{align*}
	which is the $m+1$ case of \eqref{l1}.
\end{proof}
For $n\leq (\nu +m)(\nu +m+1)/2$, equating coefficients of $q^n$ on both sides of \eqref{l1},
\begin{equation}\label{l4}
\sum_{\tau =0}^{m}(-1)^\tau p\left(n-\dfrac{(\nu +\tau )(\nu +\tau -1)}{2}\right)=N_\nu (n).
\end{equation}
Now the statement \eqref{CP1} is obvious from \eqref{l4} for sufficiently large $m$.

\section{Inequalities}\label{In}
Let us define $pp(n)$ to be the number of bipartitions of $n$, and $p_o(n)$ be the number of partitions of $n$ into odd parts. The generating function for $pp(n)$ and $p_o(n)$ are given by \[\sum_{n=0}^{\infty}pp(n)=\dfrac{1}{(q;q)^2_\infty}\]
and 
\[\sum_{n=0}^{\infty}p_o(n)=\dfrac{1}{(q;q^2)_\infty}=\dfrac{(q^2;q^2)}{(q;q)_\infty},\] respectively.

\begin{corollary}
	For $n$ odd or $\nu $ even,
	\begin{equation}\label{bpc}
	(-1)^\nu \left(pp(n)+2\sum_{\tau =1}^{\nu }(-1)^\tau  pp(n-\tau ^2)\right)\geq 0,
	\end{equation}
with strict inequality if $n\geq (\nu +1)^2$.
\end{corollary}
\begin{proof}
Using the fact that $\phi(-q)=\dfrac{(q;q)^2_\infty}{(q^2;q^2)_\infty}$ in \eqref{eq6} and after rearranging the terms, we obtain
\begin{align}\label{pp}
\dfrac{(-1)^\nu }{(q;q)^2_\infty}\left(1+2\sum_{\tau =1}^{\nu }(-1)^\tau q^{\tau ^2}\right)=\dfrac{(-1)^\nu }{(q^2;q^2)_\infty}+2B_\nu (q),
\end{align}
where
\[B_\nu (q)=\frac{q^{(\nu +1)^2}}{(q;q)_\infty(q^2;q^2)_\infty}\sum_{\tau =0}^{\infty}\frac{q^{2\tau ^2+(2\nu +3)\tau }}{(q^2;q^2)_\tau (q;q^2)_{\nu +\tau +1}}\]
has nonnegative coefficient. Since $\dfrac{1}{(q^2;q^2)_\infty}$ is an even function of $q$ with positive coefficients; every term in \eqref{pp} is nonnegative if $\nu$ is even, however if $\nu$ is odd every coeeficients of odd powers are nonnegative.
\end{proof}
\begin{corollary}
With strict inequality if $n> \nu (2\nu +3)$. We have
\item[1.]  For $n$ odd or $\nu $ even,
\begin{align}\label{p12}
(-1)^{\nu -1}\sum_{\tau =0}^{\nu -1}(-1)^\tau \big(p\left(n-\tau (2\tau +1)\right)-p\left(n-(\tau +1)(2\tau +1)\right)\big)\leq p(n-\nu (2\nu +1)).
\end{align}

\item[2.] For odd $n$,
\begin{align}
(-1)^{\nu -1}\sum_{\tau =0}^{\nu -1}(-1)^\tau \big(p_o\left(n-\tau (2\tau +1)\right)-p_o\left(n-(\tau +1)(2\tau +1)\right)\big)\leq p_o\left(n-\nu (2k+1)\right).
\end{align}
\end{corollary}
\begin{proof}
Proof is similar to that of \eqref{bpc} but instead of using \eqref{eq6}, we use \eqref{eq2}. So we omit the details here.
\end{proof}

\end{document}